\documentclass[12pt,a4paper,english]{smfart}
\usepackage[english]{babel}
\usepackage{ragged2e}
\usepackage{smfthm, mathabx}
\usepackage{stmaryrd}
\usepackage{amssymb}
\usepackage{amsmath}
\usepackage{amsthm}
\usepackage{amsfonts}
\usepackage{graphicx}
\usepackage{enumerate}
\usepackage{comment}

\usepackage{tikz-cd}

\usepackage{appendix}

\usepackage{euscript,mathrsfs}
\usepackage[utf8]{inputenc} 
\usepackage{longtable}
\usepackage{dsfont}

\usepackage[OT2,T1]{fontenc}

\usepackage[all,cmtip]{xy}
\usepackage{alltt}

\usepackage{calrsfs}

\makeatletter
\def\thm@space@setup{
  \thm@preskip=10pt 
  \thm@postskip=10pt 
}
\makeatother

\xyoption{all}

\usepackage{float}

\usepackage{fullpage}

\usepackage{color}

\newtheorem{theorem}{Theorem}

\newtheorem*{Remark}{Remark}

\newtheorem*{Definition}{Definition}

\def\Aut{{\rm Aut}}
\def\ker{{\rm ker}}

\def\Gg{{\mathfrak g}}

\def\so{{\mathfrak s}{\mathfrak{o}}}

\def\sl{{\mathfrak s} {\mathfrak l}}
\def\gl{{\mathfrak g} {\mathfrak l}}

\def\Sl{{\rm SL}}
\def\Gl{{\rm GL}}
\def\SO{{\rm SO}}
\def\OO{{\rm O}}

\title{Construction of non-odd Galois representations with large image}

\author{Donghyeok Lim}
\address{Department of Mathematics Education, Korea National University of Education, Cheongju 28173, South Korea}
\email{donghyeokklim@gmail.com}

\author{Christian Maire}
 \address{Université Marie et Louis Pasteur,  CNRS, Institut FEMTO-ST, F-25000 Besançon, France}
\email{christian.maire@univ-fcomte.fr}

\date{\today}


\thanks{This work was carried out during visiting positions held by the second author at Korea National University of Education. These visits were partially funded by the French-Korean International Research Network in Mathematics (CNRS).  The first author was supported by the National Research Foundation of Korea (NRF) grants No.
RS-2024-00462910. The second author was also
partially supported by the EIPHI Graduate School (contract "ANR-17-EURE-0002") and by the Bourgogne-Franche-Comté Region.}

\newcommand{\Q}{\mathbb{Q}}
\newcommand{\F}{\mathbb{F}}
\newcommand{\Z}{\mathbb{Z}}
\newcommand{\NN}{\mathbb{N}}

\def\Ker{{\rm Ker}}

\def\log{{\rm log}}

\def\Gal{{\rm Gal}}

\def\O{{\mathcal O}}

\def\Ll{{\mathcal L}}

\def\AA{{\mathbb A}}

\def\rk{{\rm rk}}

\def\1{{\bf 1}}

\def\L{{\rm L}}

\def\so{{\mathfrak s}{\mathfrak{o}}}

\def\sl{{\mathfrak s} {\mathfrak l}}
\def\gl{{\mathfrak g} {\mathfrak l}}

\begin{document}

\begin{abstract} In this work, we construct Galois representations with large image that are not $\Gl_r$-odd (or, equivalently, not regular at infinity). More precisely, for every prime $p\geq 3$, every integer $r\geq 2$, and every integer $a\in[(1-(-1)^r)/2,\,r-2]$ satisfying $a\equiv r \pmod 2$, we construct a continuous Galois representation $\rho: G_\Q \rightarrow \Gl_r(\Q_p)$ whose image is commensurable with $\Gl_r(\Z_p)$ and satisfies $|{\rm tr}(\rho(c))|=a$, where $c$ denotes a complex conjugation. We also obtain analogous results for the orthogonal group $\OO_r(\Q_p)$.
\end{abstract}

\maketitle

\section{Introduction}

Let $G_\Q$ be the absolute Galois group of $\Q$, and let $p$ be a prime number.
Since the second half of the twentieth century, (continuous) Galois representations $\rho: G_\Q \rightarrow \Gl_r(\Q_p)$ have played a central role in number theory, particularly when they arise from geometry. A key object is the image $\rho(c)$ of (a class of) complex conjugation~$c$. For instance, if $\rho_E$ arises from the action of $G_\Q$ on the Tate module $T_p(E)$ of an elliptic curve~$E$, then $\rho_E$ is \emph{odd}, meaning that ${\rm det}(\rho(c))=-1$.

\medskip

Let $K$ be a number field and let $G_K$ denote its absolute Galois group.  We introduce the following invariant associated with a Galois representation $\rho : G_K \rightarrow \Gl_r(\Q_p)$:
$$i(\rho)=\max_{v \ real} |{\rm tr}(\rho(c_v))|,$$
where $c_v$ denotes a complex conjugation at the real place $v$. Since the eigenvalues of $\rho(c_v)$ are $\pm1$, we have $i(\rho)\le r$ and $i(\rho)\equiv r \pmod 2$.

\medskip

For Galois representations of dimension $r>2$, one is naturally led to consider generalizations of the notion of oddness. We refer the reader to Calegari's note \cite{Calegari2}, based on lectures given at the Institut Henri Poincaré in 2010. See also Gross's unpublished note \cite{Gross}. In the present work, we focus on the following notion.

\begin{Definition} 
   A Galois representation $\rho: G_K \rightarrow \Gl_r(\Q_p)$ is said to be \textit{regular at infinity} (or $\Gl_r$-odd) if $i(\rho) \leq 1$.
\end{Definition}

Bellaiche-Chenevier \cite{Bellaiche-Chenevier}, Taylor \cite{Taylor}, Calegari  \cite{Calegari1}, Caraiani-Le Hung \cite{Caraiani-LeHung} studied the behavior of $i(\rho)$ when $K$ is totally real and $\rho$ arises from an automorphic form. In all these works, $\rho$ is regular at infinity.

In this context, Khare--Larsen \cite[Appendix A]{Khare-Larsen} investigated the relationship between regularity at infinity and geometric lifts of mod $p$ Galois representations. In particular, they suggest that any geometric lifting of a mod $p$ Galois representation that is regular at all places above~$p$ (i.e. distinct Hodge-Tate weights for all places above $p$) should also be regular at infinity.

\medskip

On the other hand, one may ask how to construct Galois representations with open image. Greenberg \cite{Greenberg} initiated a Galois-theoretic approach to this question for representations valued in $\Gl_r(\Z_p)$ with $r\ge3$, motivated in part by the difficulty of producing such representations by geometric means. This line of research was subsequently developed by Ray \cite{Ray_2} \cite{Ray}, Cornut-Ray \cite{Cornut-Ray}, Maire \cite{Maire}, Tang \cite{Tang}, etc.
In particular, \cite{Maire} proves that for every $p>2$ and every $r\geq 1$, there exists a Galois representation $\rho: G_\Q \rightarrow \Gl_r(\Z_p)$ with open image. These constructions rely on group-theoretic methods rather than automorphic ones, although it seems nontrivial to prove that the resulting representations are not geometric.

Meanwhile, Katz \cite{Katz} gave an alternative construction of Galois representations with open image in $\Gl_r(\Z_p)$, for $r\geq 6$ and $p\equiv 1 \pmod 3$ or $p \equiv 1 \pmod 4$, using the Jacobian varieties of some explicit curves. However, in all the cases mentioned above, the resulting Galois representations are regular at infinity.

\medskip

In this work, we construct Galois representations with large image that are not regular at infinity.
\medskip

Let $\omega : G_\Q \rightarrow \Gl_1(\Z_p)$ be the cyclotomic character.
For an integer $r$, set $\epsilon(r)=\frac{1-(-1)^r}{2}$.

\smallskip

\begin{theorem}\label{theo-sl}
For every prime $p\geq 3$, every $r\geq 2$, and every integer $a$ satisfying $\varepsilon (r) \leq a \leq   r-2$ and $a\equiv r \ ({\rm mod} \ 2)$, there exists a continuous Galois representation $\rho : G_\Q \rightarrow \Gl_r(\Z_p)$ unramified outside $\{p,\infty\}$ if $p\equiv 3$ mod $4$ (resp. outside $\{2,p,\infty\}$ if $p\equiv 1 $ mod $4$) such that:
\begin{itemize}
\item[$(1)$]  $ i(\rho)=a$;
\item[$(2) $] The image of $\rho$ and $\Sl_r(\Z_p)$ are commensurable. More precisely, let $k, m,  n \in \NN$ with  $ m\geq n\geq 1$ be such that 
\begin{equation*}
 m+n=r, \quad m-n=a,  \quad p^k > m(m+2)(m+3), 
\end{equation*}
then $\rho(G_\Q)$ contains the principal congruence subgroup $\Sl_r^{4kmn-2k+2}(\Z_p)$ as an open subgroup;
\item[$(3)$] The Galois representation $\rho\otimes \omega : G_\Q \rightarrow \Gl_r(\Z_p)$  has open image with $i(\rho\otimes \omega)=a$.
\end{itemize} 
\end{theorem}

Recall that 
 two topological groups $G$ and $H$ are said to be  {\it commensurable} if they have a common open subgroup.

\smallskip

In the same spirit as Cornut-Ray \cite{Cornut-Ray},  Tang  \cite{Tang}, and Maletto \cite{Maletto}, we obtain the following result by varying the target group.

\begin{theorem}\label{theo-so}
For every $p\geq 3$, for every $r\geq 2$, for every $a$ with  $\varepsilon (r) \leq a \leq   r-2$ and $a\equiv r \ ({\rm mod} \ 2)$, there exist a continuous Galois representation $\rho : G_\Q \rightarrow \OO_r(\Z_p)$ in the orthogonal group $\OO_r(\Z_p)$,  unramified outside $\{p,\infty\}$ (resp. outside $\{2,p,\infty\}$) if $p\equiv 3$ mod $4$ (resp. $p\equiv 1 $ mod $4$) and such that:
\begin{itemize}
\item[$(1)$] $i(\rho)=a$;
\item[$(2)$] The image of $\rho$ is open. More precisely, let $k, m,  n \in \NN$ with  $ m\geq n\geq 1$ be such that
\begin{equation*} m+n=r, \quad m-n=a, \quad 
p^k > \frac{3n(m+2)}{2} +m, 
\end{equation*}
then $\rho(G_\Q)$ contains the principal congruent subgroup $\SO_r^{2kmn-2k+2}(\Z_p)$ of $\SO_r(\Z_p)$.
\end{itemize}
\end{theorem}

\begin{Remark} We prove a slightly more general result than those stated above. More precisely, let $p\geq 3$ and let  $K$ be an imaginary quadratic field such that $p\nmid h_K$, where $h_K$ denotes the class number of $K$. Then there exists a Galois representation $\rho: G_\Q \rightarrow \Gl_r(\Z_p)$ satisfying conditions  $(1)-(2)$ of Theorem \ref{theo-sl} (resp. Theorem~\ref{theo-so}),  and such that  $\rho$ factors through the maximal pro-$p$ extension of $K$ unramified outside~$p$.  In particular, the ramification set of $\rho$ consists of~$p$ together with the primes ramified in  $K/\Q$.   
\end{Remark}

This paper is divided into three sections. The next section is devoted to the strategy developed in \cite{Maire}, together with some Lie algebra-theoretic background. Its main result is Theorem \ref{theorem_4.3_bis}, which relies on Kuranishi's theorem stating that every semisimple Lie algebra over a field of characteristic $0$, in particular over $\Q_p$, is generated by two elements $x$ and $y$. To obtain our main results, however, we require a refinement of this theorem. More precisely, fix a partition $r=m+n$ with $m \geq n \geq 1$, and let $A$ be the $r \times r$ diagonal matrix whose first $m$ diagonal entries are $1$ and whose last $n$ diagonal entries are $-1$.
In Section \ref{section_Lie_algebra}, we refine Kuranishi's theorem for the semisimple Lie algebras $\sl_r(\Q_p)$ and $\so_r(\Q_p)$  by taking into account the involutive automorphism induced by conjugation by~$A$.

The final section is devoted to the proofs of the two theorems stated in the introduction.


\section{Group-theoretic tools for lifting Galois representations}

For a finitely generated pro-$p$ group~$N$, denote by $N^{p,el}:=N/N^p[N,N]$ its maximal $p$-elementary quotient. Let $\pi_N: N\rightarrow N^{p,el}$ be the natural projection.

\smallskip

The starting point of our strategy is the following lifting result.

 \begin{theo} \label{theorem_lifting}
 Let $\Gamma \rtimes \Delta $  be a profinite group, where $\Gamma$ is a finitely generated free pro-$p$ group and $\Delta \subset \Aut(\Gamma)$ is a finite group of order coprime to $p$. Let $H$ be a pro-$p$ group equipped with an action of a finite group $\Delta'$, where $\Delta' \cong \Delta$. Fix an isomorphism $\rho_0 : \Delta \xrightarrow{\cong} \Delta'$ and regard $H$ as a $\Delta$-group vis $\rho_0$. Suppose that the $\Delta$-module $H^{p,el}$ is isomorphic to a sub-$\Delta$-module of $\Gamma^{p,el}$. Let $\alpha : \Gamma  \rtimes  \Delta \twoheadrightarrow H^{p,el}\rtimes \Delta'$ be the homomorphism induced by a surjective $\Delta$-homomorphism $\Gamma^{p,el}\twoheadrightarrow H^{p,el}$ and $\rho_0$. Let $\pi:H\rtimes\Delta'\twoheadrightarrow H^{p,el}\rtimes\Delta'$ be the homomorphism determined by $\pi|_H=\pi_H$ and $\pi|_{\Delta'}=\mathrm{id}_{\Delta'}$.
 Then the embedding problem
   $$\xymatrix{   & \Gamma\rtimes \Delta \ar@{.>>}[ld]_{ \psi} \ar@{->>}[d]^-{\alpha} \\
  H\rtimes \Delta'  \ar@{->>}[r]_-\pi &  H^{p,el}  \rtimes \Delta' }$$
 admits a proper continuous solution $\psi$.
 \end{theo}
 Since we start with a free pro-$p$ group $\Gamma$, the lifting problem is simpler than the one considered in \cite{Maire}. In particular, this theorem is proved in Greenberg \cite[Proposition 2.3.1]{Greenberg}. It also appears in \cite[Proposition 2.14]{HM}, where it is deduced from the unpublished work of Wingberg \cite{Wingberg}. Since this result is crucial for our purposes, we choose to give the main arguments of the proof, following \cite{HM} and \cite{Wingberg}.

\begin{proof}
The result is immediate if $\Delta=1$. For the general case, by assumption, there exists an $\F_p[\Delta]$-module isomorphism $H^{p,el} \oplus M \cong \Gamma^{p,el}$ for some $M$. Hence, replacing $H\rtimes\Delta'$ by $(H\times M)\rtimes\Delta'$ and modifying the chosen homomorphism $\alpha$ accordingly, if necessary, we may assume that $\Gamma$ and $H$ have the same generator rank. Let $\iota_1 : \Delta \to \Aut(\Gamma)$ denote the homomorphism corresponding to the original action of $\Delta$ on $\Gamma$, and henceforth write the semidirect product $\Gamma\rtimes\Delta$ appearing in the statement as $\Gamma\rtimes_{\iota_1}\Delta$.

Let $\alpha' : \Gamma \to H$ be a lift of the restriction $\alpha \mid_{\Gamma}$ of $\alpha$ to $\Gamma$. By \cite{Wingberg} or \cite[Lemma 2.15]{HM}, $\Gamma$ admits a $\Delta$-action defined by a homomorphism $\iota_2 : \Delta \to \Aut(\Gamma)$ for which $\alpha'$ is $\Delta$-equivariant. Let $\Gamma \rtimes_{\iota_2} \Delta$ be the semidirect product associated with $\iota_2$, and let $\beta : \Gamma \rtimes_{\iota_2} \Delta \to H \rtimes \Delta'$ be the surjection induced by $\alpha'$ and $\rho_0$.

The actions defined by $\iota_1$ and $\iota_2$ induce the same action on $\Gamma^{p,el}$. Thus, by the Schur-Zassenhaus theorem, there exists $g \in \Ker(\Aut(\Gamma) \to \Aut(\Gamma^{p,el}))$ such that $g \circ \iota_1(s)=\iota_2(s) \circ g$ for every $s \in \Delta$ (cf. \cite[Proposition 2.13]{HM}). It is straightforward to check that the composite of $\beta$ with the homomorphism
\begin{equation*}
    \Gamma \rtimes_{\iota_1} \Delta \rightarrow \Gamma \rtimes_{\iota_2} \Delta \qquad (\gamma, s) \to (g(\gamma), s)
\end{equation*}
is a solution to the embedding problem.
\end{proof}

We now apply the above lifting theorem in the setting of imaginary quadratic fields.


 \subsection{The lifting theorem and imaginary quadratic fields} \label{section_imaginary}
  Take $p\geq 3$.
  We begin with an imaginary quadratic extension $K/\Q$. 
  Put $\Delta=Gal(K/\Q)=\langle s \rangle$. Let $\1$ denote the trivial $\F_p$-character of $\Delta$, and let $\varphi$ be its nontrivial character. 
  Let~$K_p$ be the maximal pro-$p$ extension of $K$ unramified outside $p$, and set $G_{K,p}=Gal(K_p/K)$.
Observe that $\Delta$ acts on~$G_{K,p}$ and hence on~$G_{K,p}^{p,el}$. We denote by $\chi(G_{K,p}^{p,el})$ the character of the $\F\left[\Delta\right]$-module $G_{K,p}^{p,el}$.

The following result is very well-known (see for example \cite[Example 1.10]{Maire}; the case $p=3$ is also obvious).

\begin{prop} \label{prop_imaginaryquadratic} Let $K$ be an imaginary quadratic field  such that $p\nmid h_K$. The  group $G_{K,p}$ is free pro-$p$ on two generators and $\chi(G_{K,p}^{p,el})=\1+ \varphi$.
\end{prop}

\begin{exem} \label{example} Set $K=\Q(\sqrt{-p})$.
  Thanks to an explicit version of Brauer-Siegel Theorem (see for example \cite{Louboutin}), $p\nmid h_K$,  and therefore $G_{K,p}$ satisfies Proposition \ref{prop_imaginaryquadratic}.
\end{exem}

For our purposes, let us work in the setting of Proposition \ref{prop_imaginaryquadratic}.
  Let  $G$ be a pro-$p$ group having an automorphism  $t$ of order $2$. Set $\Delta'=\langle t \rangle$, and consider the isomorphism $\rho_0: \Delta \stackrel{\sim}{\rightarrow} \Delta'$.
   Suppose that $g_1, g_2 \in G$ satisfy $t(g_1)=g_1$ and $t(g_2)=g_2^{-1}$. Let $H:=\langle g_1, g_2 \rangle$ be the closed subgroup of $G$ generated by $g_1$ and~$g_2$. Then $\Delta' $ is a subgroup  of $ \Aut(H)$, and  $\Delta$ acts on $H$  via $\rho_0$. The character $\chi(H^{p,el})$ of the $\F_p\left[\Delta\right]$-module $H^{p,el}$ is given  by $\chi(H^{p,el}) = \1+\varphi$.
As consequence of  Theorem \ref{theorem_lifting} and Proposition \ref{prop_imaginaryquadratic} applied to $\Gamma=G_{K,p}$, we obtain the following surjective morphism $$\psi:Gal(K_p/\Q) \simeq G_{K,p} \rtimes \Delta \twoheadrightarrow H \rtimes \Delta'.$$
  Let $F$ be the subfield of $K_p$ fixed by the kernel of $\psi$. Then $F/\Q$ is a Galois extension such that $Gal(F/\Q) \simeq  H \rtimes Gal(K/\Q)$. Thus, we have proved the following.

 \begin{theo} \label{theorem_4.3} 
 Let $p\geq 3$, and let $K/\Q$ be an imaginary quadratic extension with
   $p\nmid h_K$. Let $G$ be a pro-$p$ group equipped with an automorphism $t$ of order $2$. Assume that there exist $g_1, g_2 \in G$ such that  $t(g_1)=g_1$ and $t(g_2)=g_2^{-1}$, and set $H=\langle g_1, g_2\rangle$. 
  Then there exist a tower of  Galois extensions $K_p/F/K/\Q$ such that $Gal(F/\Q) \simeq  H \rtimes Gal(K/\Q)$. 
  \end{theo}

\subsection{$p$-adic analytic pro-$p$ group}

In this section, we revisit the key ingredients from \cite{Maire} concerning $p$-adic analytic groups. We shall apply Theorem \ref{theorem_lifting} with a $p$-adic analytic pro-$p$ group $H$, exploiting the correspondence between $\Z_p$-Lie algebras and $p$-adic analytic groups. We refer the reader to \cite{DDMS} for further details.

\smallskip

Take $p\geq 3$ and $r\geq 2$. Denote by $\Gl_r(\Z_p)$ the general linear group over $\Z_p$, by $\Sl_r(\Z_p)$ its subgroup of matrices of determinant $1$. Let $O_r(\Z_p)$ be the subgroup of $\Gl_r(\Z_p)$ consisting of matrices
$M$ satisfying $M^tM=I$, and set $\SO_r(\Z_p):=O_r(\Z_p)\cap \Sl_r(\Z_p)$.

For each integer $k\geq 1$, let $\varphi_k : \Gl_r(\Z_p) \rightarrow \Gl_r(\Z/p^k\Z)$ denote the natural reduction map. We then define the principal congruence subgroups by
\begin{equation*}
    \Gl_r^k(\Z_p):=\ker(\varphi_k), \quad \Sl_r^k(\Z_p):= \ker(\varphi_k) \cap \Sl_r(\Z_p), \quad \SO_r^k(\Z_p):= \ker(\varphi_k) \cap \SO_r(\Z_p).
\end{equation*}

\subsubsection{The $\Z_p$-Lie algebra $\gl_r$}

Let $\gl_r(\Z_p)$ denote the free $\Z_p$-module of rank~$r^2$ with basis $\{E_{i,j}\}_{1 \leq i,j \leq r}$, where $E_{i,j}$ is the $(i,j)$th elementary matrix.
Then $\gl_r$ becomes a $\Z_p$-Lie algebra with Lie bracket $[A\ B]=AB-BA$. Set $\gl_r(\Q_p):=\Q_p\otimes_{\Z_p} \gl_r$, the $\Q_p$-Lie algebra of all $r\times r$ matrices over $\Q_p$ with the same Lie bracket. For $k\geq 1$, let $\gl_r^k:=p^k\gl(\Z_p)$.
Since $[\gl_r^k \,\, \gl_r^k] \subset p \gl_r^k$, $\gl_r^k$ is a  {\it powerful} $\Z_p$-Lie algebra (see \cite[Chapter 9, \S 9.4]{DDMS}).

\medskip

The exponential and logarithmic maps $\exp:\gl_r^1\to\Gl_r^1(\Z_p)$ and
$\log:\Gl_r^1(\Z_p)\to\gl_r^1$ are given by $\exp(x)=\sum_{n\ge0}\frac{x^n}{n!}$ and $\log(z)=\sum_{n\ge1}\frac{(-1)^{n+1}}{n}(z-1)^n$. Recall that $\exp$ and $\log$ are mutually inverse and induce an order-preserving bijection
\begin{equation*}
    L \longmapsto \exp(L), \qquad G \longmapsto \log(G)
\end{equation*}
between powerful $\Z_p$-Lie subalgebras $L$ of $\gl_r^1$ and uniform subgroups $G$ of $\Gl_r^1(\Z_p)$ (\cite[Theorem 9.10]{DDMS}). It is well-known that $\exp(\gl_r^k)=\Gl_r^{k}(\Z_p)$ and that $\Gl_r^k(\Z_p)$ is {\it uniform} for each $k$ (cf. \cite[Theorem 5.2, \S 4.1]{DDMS}).

\medskip

Let $G \subset \Gl_r^1(\Z_p)$ be a $p$-adic analytic group. Then $G$ contains an open uniform subgroup~$G_0$. Let $\Gg_0:=\log(G_0)$. Then $\Gg_0$ is a powerful $\Z_p$-Lie subalgebra of $\gl_r^1$. 
The $\Q_p$-Lie algebra associated with $G$ is defined as $\Gg_0(\Q_p):=\Q_p \otimes_{\Z_p} \Gg_0$. This $\Q_p$-Lie algebra is independent of the choice of $G_0$. Thus, we may unambiguously write $\Gg(\Q_p):=\Gg_0(\Q_p)$.

\medskip

\subsubsection{The $\Z_p$-Lie algebras $\sl_{r}$ and $\so_{r}$ }
Let $\sl_r(\Z_p)$ denote the $\Z_p$-Lie subalgebra of $\gl_r(\Z_p)$ consisting of trace-zero matrices. For each integer $k\geq 1$, set $\sl_r^k:=p^k\sl_r(\Z_p)=\gl_r^k\cap \sl_r(\Z_p)$. It follows from $\exp(\sl^1_r(\Z_p))=\Sl_r(\Z_p)$ that $\exp(\sl_r^k)=\Sl_r^k(\Z_p)$. Since $\sl_r^k$ is powerful, $\Sl_r^k(\Z_p)$ is uniform pro-$p$ of dimension $r^2-1$. This equals the $\Q_p$-dimension of the $\Q_p$-Lie algebra $\sl_r(\Q_p)=\Q_p\otimes_{\Z_p}\sl_r(\Z_p)$.

\medskip

Next, let $\so_r(\Z_p)$ denote the $\Z_p$-Lie subalgebra of $\gl_r(\Z_p)$ consisting of matrices $M$ satisfying $M+M^t=0$. For each integer $k\geq 1$, set $\so_r^k:= p^k\so_r(\Z_p)$. The  algebra $\so_r^k$ is powerful and consequently its exponential $\SO_r^k(\Z_p)$ is a uniform pro-$p$ group. Its dimension is $\frac{r(r-1)}{2}$, equal to the $\Q_p$-dimension of the $\Q_p$-Lie algebra $\so_r(\Q_p)=\Q_p \otimes_{\Z_p} \so_r(\Z_p)$.

We shall write $\sl_r$ (resp. $\so_r$) for $\sl_r(\Z_p)$ (resp. $\so_r(\Z_p)$) whenever no confusion can arise.

\subsubsection{Kuranishi pair} \label{section_simplealgebra}

The $\Q_p$-Lie algebras $\sl_r(\Q_p)$ and $\so_r(\Q_p)$ are semisimple. The following result is essential to our strategy.

\begin{theo}[Kuranishi \cite{Kuranishi}] \label{theo_kuranishi}
 A semisimple $\Q_p$-Lie algebra $\Ll$ can be  generated by $2$ elements.
\end{theo}

As a corollary of Theorem \ref{theo_kuranishi} we get

\begin{coro} \label{coro_locallyequal}
 Let $G\subset \Gl_r^1(\Z_p)$ be a $p$-adic analytic group such that $\Gg(\Q_p)$ is semisimple. Then there  exist two elements
 $g_1$ and $g_2$ in $G$  such that $G$ and the (closed) subgroup~$H$ generated by  $g_1$ and $g_2$ are commensurable.
\end{coro}

The proof is straightforward, but since this observation is essential, we include it.

\begin{proof} Let $G_0$ be an open uniform subgroup of $G$.
Let $\Gg_0:=\log(G_0)$ be the powerful $\Z_p$-Lie subalgebra of $\gl_r^1$ associated to $G_0$, and set $\Ll:=\Gg_0(\Q_p)=\Q_p\otimes_{\Z_p} \Gg_0$. 
By Theorem \ref{theo_kuranishi}, there exist $z_1,z_2 \in \Ll$ such that $\Ll=\langle z_1,z_2 \rangle$. 
After multiplying $z_1$ and $z_2$ by some powers of~$p$, we can assume that $z_1$ and $z_2$ also belong to $\Gg_0$.
Set $g_1=\exp(z_1)$ and $g_2=\exp(z_2)$, and let  $H=\langle g_1,g_2\rangle$ be the closed subgroup of $G_0$ generated by $g_1$ and $g_2$. Since $H$ is a closed subgroup of a $p$-adic analytic group $G_0$, it is itself $p$-adic analytic. Let $H_0$ be an open uniform subgroup of $H$. Then for $n\gg 0$, $g_1^{p^n}$ and $g_2^{p^n}$ are in $H_0$. Hence the $\Z_p$-Lie algebra $\Ll_{H_0}:=\log(H_0)$  contains both $p^nz_1$ and $p^nz_2$, and then $\Q_p\otimes_{\Z_p} \Ll_{H_0}=\Ll$. Thus, $H$ and $G$ are commensurable. \end{proof}

For what follows, we use the following terminology. Let $G$ (resp. $\Gg$) be a $p$-adic analytic group (resp. a $\Z_p$-Lie algebra).  A pair $(g_1, g_2)$ of elements of $G$ (resp. $(z_1,z_2)$ of $\Gg$)  is called a \emph{Kuranishi pair} of $G$ (resp. of $\Gg$)  if the closed subgroup generated by $g_1$ and $g_2$ (resp. the $\Z_p$-Lie subalgebra generated by $z_1$ and $z_2$) is open in $G$ (resp. of finite index in $\Gg$).


   \subsection{A key principle} 
   \label{section_liealgebra_uniform}

Now we explain the principle underlying our approach. We continue to work in the setting of the previous subsection. For a more general treatment, we refer the reader to \cite{Maire}.
Let $G$ be a $p$-adic analytic group whose associated $\Q_p$-Lie algebra $\Gg(\Q_p)$ is semisimple. Let $G_0$ be an open uniform subgroup of $G$, and let $(z_1, z_2)$ be a Kuranishi pair of $\Gg_0=\log(G_0)$. Set $g_i=\exp(z_i)$ for $i=1,2$ and let $H=\langle g_1, g_2 \rangle$. Then $H$ and $G$ are commensurable.

Assume moreover that there exists an element $A\in \Gl_r(\Z_p)$ of order $2$ satisfying $Az_1A^{-1}=z_1$ and $Az_2A^{-1}=-z_2$. By the equivariance of the exponential map with respect to conjugation, it follows that $A g_1A^{-1}=g_1$ and $Ag_2A^{-1}=g^{-1}_2$. 
Hence the conjugation map $t_A: g\mapsto AgA^{-1}$ defines an automorphism of $H$ of order $2$. 
Combining Theorem \ref{theorem_4.3} with the above discussion of Kuranishi pairs, we obtain the following.

 \begin{theo} \label{theorem_4.3_bis} Assume that $p\geq 3$ and $r\geq 2$. 
  Let $K/\Q$ be an imaginary quadratic field extension such that
   $p\nmid h_K$. Let $G\subset \Gl_r(\Q_p)$ be a  $p$-adic analytic  group. Suppose that there exists a Kuranishi pair $(g_1, g_2)$ of $G$ and a matrix $A \in \Gl_r(\Z_p)$ of order $2$ such that $A g_1A^{-1}=g_1$ and $A g_2 A^{-1}=g_2^{-1}$.
  Then there exists a continuous Galois representation $\rho: Gal(K_{p}/\Q)\rightarrow \Gl_r(\Q_p)$ 
  such that $\mathrm{im}(\rho)$ and $G$ are commensurable. Moreover, $i(\rho)=|{\rm tr}(A)|$.
  \end{theo}

  \begin{proof}  
  Set $\Delta':=\langle A \rangle $ and $\Delta=Gal(K/\Q)=\langle s\rangle $. Let $\rho_0: \Delta \stackrel{\simeq}{\rightarrow} \Delta'$. Now observe that $\Delta$ acts on $H=\langle g_1, g_2 \rangle$ and the  character of the $\Delta$-module $H^{p,el}$ is given by  $\chi(H^{p,el})=\1+\varphi$. By Theorem \ref{theorem_4.3}, we get a surjective morphism $\rho: G_{K,p} \rtimes \Delta \twoheadrightarrow  H \rtimes \Delta'$, such that $\rho(s)=A$. We conclude by noting that $H\rtimes \Delta'\hookrightarrow \Gl_r(\Q_p)$. 
   
   Let $c$ denote the (conjugacy class of) complex conjugation. By the Schur-Zassenhaus theorem, $A$ and $\rho(c)$ belong to the same conjugacy class. Hence, we have $i(\rho)=|{\rm tr}(A)|$.\end{proof}

Let $G$ be a uniform subgroup of $\Gl_r(\Z_p)$ whose associated  $\Q_p$-Lie algebra $\Gg(\Q_p)$ is semisimple. To apply Theorem \ref{theorem_4.3_bis} with $i(\rho)=a$, we therefore need to find a Kuranishi pair $(z_1,z_2)$ in the $\Z_p$-Lie algebra $\Gg:=\log(G)$ and a matrix~$A \in \Gl_r(\Z_p)$ of order $2$ satisfying  $|{\rm tr}(A)|=a$ such that conjugation by $A$ acts as $+1$ on $z_1$ and as $-1$ on $z_2$. In the next section, we verify that these conditions are satisfied for the $\Z_p$-Lie algebras $\sl_r$ and $\so_r$.


\section{Lie algebra generators compatible with the involution} \label{section_Lie_algebra}

Let $r \geq 2$ be an integer. We fix a partition $r=m+n$ with $m \geq n \geq 1$, and let $d:=d_{m,n}$ denote the $r \times r$ diagonal matrix whose first $m$ diagonal entries are $1$ and whose last $n$ diagonal entries are $-1$. In this section, we refine Theorem \ref{theo_kuranishi} in the cases of $\sl_r(\Q_p)$ and $\mathfrak{so}_r(\Q_p)$ by taking into account the involutive automorphism given by conjugation by $d$. This leads to the following result. Throughout this section, $\mathfrak{sl}_r$ and $\mathfrak{so}_r$ denote $\mathfrak{sl}_r(\Z_p)$ and $\mathfrak{so}_r(\Z_p)$, respectively. 

\begin{prop}\label{prop-Lie generator}
Let $\mathfrak{g}$ be either $\sl_r$ or $\so_r$, and let $\mathfrak{g}^+$ and $\mathfrak{g}^-$ denote the subspaces of $\mathfrak{g}$ on which the involution $x \to dxd^{-1}$ acts as multiplication by $1$ and $-1$, respectively.
\begin{enumerate}
\item[$(1)$] There exist $\alpha \in \mathfrak{g}^+$ and $\beta \in \mathfrak{g}^-$ such that the $\Z_p$-Lie subalgebra $\langle \alpha, \beta \rangle$ of $\Gg$ has the same $\Z_p$-rank as $\Gg$.
\item[$(2)$] Assume that $k \in \NN$ satisfies $p^k > m(m+2)(m+3)$. Then the elements $\alpha \in \mathfrak{sl}_r^+$ and $\beta \in \mathfrak{sl}^-_r$ may moreover be chosen so that $\langle \alpha, \beta \rangle \supseteq \mathfrak{sl}^{2(k-1)(2mn-1)}_r$. 
\item[$(3)$] Assume that $k \in \NN$ satisfies $p^k > 6(\lfloor m/2 \rfloor+1)\cdot \lfloor n/2 \rfloor+2 \lfloor m/2 \rfloor$. Then the elements $\alpha \in \mathfrak{so}^+_r$ and $\beta \in \mathfrak{so}_r^-$ may moreover be chosen so that $\langle \alpha, \beta \rangle \supseteq \mathfrak{so}^{2(k-1)(mn-1)}_r$.
\end{enumerate}
\end{prop}

Proposition \ref{prop-Lie generator} implies that if $p$ is sufficiently large relative to $r$, then each of $\mathfrak{so}_r$ and $\mathfrak{sl}_r$ admits two generators, one lying in the $1$-eigenspace and the other in the $-1$-eigenspace.

The following variant of Proposition \ref{prop-Lie generator} is obtained by slightly modifying its proof and will be used in the proofs of the main theorems presented in the next section.

\begin{prop}\label{prop - Kuranashi - p}
\begin{itemize}
\item[$(1)$] Assume that $k \in \NN$ satisfies $p^k > m(m+2)(m+3)$. Then there exist $\alpha \in p\mathfrak{sl}^+_r$ and $\beta \in p\mathfrak{sl}^{-}_r$ such that $\langle \alpha, \beta \rangle \supseteq \mathfrak{sl}^{4mnk-2k+2}_r$.
\item[$(2)$] Assume that $k \in \NN$ satisfies $p^k > 6(\lfloor m/2 \rfloor+1)\cdot \lfloor n/2 \rfloor+2 \lfloor m/2 \rfloor$. Then there exist $\alpha \in p\mathfrak{so}^+_r$ and $\beta \in p\mathfrak{so}^{-}_r$ such that $\langle \alpha, \beta \rangle \supseteq \mathfrak{so}_r^{2mnk-2k+2}$.
\end{itemize}
\end{prop}

We first introduce some notation. For $\mathfrak{g}=\sl_r$ or $\so_r$, every element $x \in \mathfrak{g}$ admits a block decomposition
\begin{equation*}
x=\begin{pmatrix}
    A & B_1 \\
    B_2 & C
\end{pmatrix}, \quad A \in M_m(\Z_p),\quad B_1 \in M_{m,n}(\Z_p), \quad B_2 \in M_{n,m}(\Z_p), \quad C \in M_n(\Z_p).
\end{equation*}
A direct computation shows that $\mathfrak{g}^+$ consists precisely of those elements with $
B_1=0$ and $B_2=0$, whereas $\mathfrak{g}^-$ consists precisely of those for which $A=0$ and $C=0$. Accordingly, we write elements of $\Gg^+$ and $\Gg^-$ as $(A,C)^{\! \scriptscriptstyle +}$ and $(B_1, B_2)^{\! \scriptscriptstyle -}$, respectively.

For $\mathfrak{g}=\mathfrak{so}_r$, we further have $B_1^t=-B_2$. Accordingly, for $B \in M_{m , n}(\Z_p)$, we write $(B)_{\mathfrak{so}}$ for the element $(B, -B^t)^{\! \scriptscriptstyle -} \in \mathfrak{g}^-$. Since conjugation by $d$ is a Lie algebra automorphism of $\Gg$, we have $[\Gg^- \,\, \Gg^-] \subset \Gg^+$. Moreover, each $x\in \mathfrak{g}^+$ induces a $\Z_p$-linear endomorphism $\mathrm{ad}_x : \mathfrak{g}^- \to \mathfrak{g}^-$ of $\mathfrak{g}^-$ defined by $\mathrm{ad}_x(y):=[x \,\,y]$. We also note that $\rk_{\Z_p}$, which denotes the $\Z_p$-rank, satisfies $\rk_{\Z_p} \mathfrak{sl}_r^{-} = 2mn$ and $\rk_{\Z_p} \mathfrak{so}^-_r = mn$.

\begin{lemm}\label{lemm- conjugation}
For $\mathfrak{g}=\mathfrak{sl}_r$ or $\mathfrak{so}_r$, we have $\mathfrak{g}^+=[\mathfrak{g}^- \, \mathfrak{g}^-]$.
\end{lemm}

\begin{proof}
It suffices to show that $[\Gg^- \,\,\Gg^-]$ contains a $\Z_p$-basis of $\Gg^+$. For $\Gg=\sl_r$, this follows from the identities
\begin{equation*}
  (E_{i,j},0)^{\! \scriptscriptstyle +} = [(E_{i, 1}, 0)^{\! \scriptscriptstyle -} \,\, (0, E_{1,j})^{\! \scriptscriptstyle -}], \qquad   (0,E_{i,j})^{\! \scriptscriptstyle +}=[(0, E_{i,1})^{\! \scriptscriptstyle -} \,\, (E_{1,j}, 0)^{\! \scriptscriptstyle -}]
\end{equation*}
for $i \neq j$, together with the identities $(E_{i,i}, -E_{j,j})^{\! \scriptscriptstyle +}=[(E_{i,j},0)^{\! \scriptscriptstyle -} \ (0,E_{j,i})^{\! \scriptscriptstyle -}]$ for $1 \leq i \leq m$ and $1 \leq j \leq n$. Similarly, the claim for $\mathfrak{so}_r$ follows from the identities
\begin{equation*}
    (E_{i,j}-E_{j,i},0)^{\! \scriptscriptstyle +} = [ (E_{i,1})_{\mathfrak{so}}\,\, (-E_{j,1})_{\mathfrak{so}}], \qquad (0, E_{i,j}-E_{j,i})^{\! \scriptscriptstyle +} = [-(E_{1,i})_{\mathfrak{so}} \,\, (E_{1,j})_{\mathfrak{so}}],
\end{equation*}
for $i \neq j$.
\end{proof}
We isolate the following technical lemma.
\begin{lemm}\label{lemm - eigenvalues}
\begin{enumerate}
\item[$(1)$] If $p^k > m(m+2)(m+3)$, there exist elements $\{\lambda_i\}_{1 \leq i \leq m}$ and $\{ \mu_j\}_{1 \leq j \leq n}$ in $\Z_p$ such that $\sum_{i=1}^m \lambda_i + \sum_{j=1}^n \mu_j=0$ and the $2mn$ elements of $\pm (\lambda_i - \mu_j)$ are pairwise distinct modulo $p^k$.
\item[$(2)$] If $p^k > 6(\lfloor m/2 \rfloor+1) \cdot\lfloor n/2 \rfloor+2\lfloor m/2 \rfloor$, then there exist elements $\{\lambda_i\}_{1 \leq i \leq \lfloor m/2 \rfloor}$ and $\{ \mu_j\}_{1 \leq j \leq \lfloor n/2 \rfloor}$ in $\Z_p$ such that the elements $ \pm \lambda_i \pm \mu_j$,  $\pm \mu_j$, $\pm\lambda_i$, $0$ are all distinct modulo $p^k$.
\end{enumerate}
\end{lemm}

\begin{proof}
Recall that $m \geq n$. To prove $(1)$, set $\lambda_i=-(n+1)i - 1$ for $1 \leq i \leq m$, $\mu_j=j$ for $1 \leq j \leq n-1$, and $$\mu_n = -\sum_{i=1}^m \lambda_i - \sum_{j=1}^{n-1} \mu_j = \frac{(n+1)m(m+1)}{2} + m - \frac{n(n-1)}{2}.$$
A direct computation shows that
\begin{align*}
  |\lambda_i-\mu_j| \leq \mu_n - \lambda_m & \leq \frac{m(m+1)(n+1)}{2} + m + (n+1)m  \\
  & \leq \frac{m(m+2)(m+1)}{2} + m(m+2) < \frac{p^k}{2},
\end{align*}
and hence $\epsilon (\lambda_i- \mu_j) \equiv \epsilon'(\lambda_{i'}-\mu_{j'}) \pmod{p^k}, \  \epsilon, \epsilon' \in \{\pm1\}$, if and only if $\lambda_i-\mu_j=\lambda_{i'}-\mu_{j'}$. It therefore remains to show that $\lambda_i - \lambda_{i'}=\mu_j - \mu_{j'}$ implies $i=i'$ and $j=j'$. This is immediate, since $|\mu_j- \mu_{j'}|$ is either less than $n$ or is at least
\begin{align*}
   \mu_n - \mu_{n-1} & = \frac{(n+1)m(m+1)-n(n-1)}{2} + m -(n-1) \\
                             & > \frac{nm(m+1)}{2} > (n+1)(m-1) = |\lambda_m-\lambda_1| \geq |\lambda_i - \lambda_{i'}|.
\end{align*}

To prove $(2)$, set $\mu_j=(\lfloor m/2 \rfloor+1) \cdot j$ for $1 \leq j \leq \lfloor n/2 \rfloor $ and $\lambda_i= 2(\lfloor m/2 \rfloor+1) \cdot \lfloor n/2 \rfloor + i$ for $1 \leq i \leq \lfloor m/2 \rfloor$. Since $|\lambda_i|+|\mu_j| < p^k/2$, two elements among $ \pm \lambda_i \pm \mu_j$,  $\pm \mu_j$, $\pm\lambda_i$, $0$ can be congruent modulo $p^k$ only if they are equal. Moreover, the inequality $|\lambda_i| > 2|\mu_j|$ implies that $\pm\lambda_i \pm \mu_j \neq \mu_{j'}$ and $\pm \lambda_i \neq \pm \mu_j$. Finally, since $|\lambda_i - \lambda_{i'}| < \lfloor m/2 \rfloor+1$, we have $\pm \lambda_i \pm \mu_j \neq \pm \lambda_{i'}$ and the elements $\pm \lambda_i \pm \mu_j$ are pairwise distinct.
\end{proof}

\begin{lemm}\label{lemm - Vandermonde} 
Let $\mathrm{M} \in M_d(\Z_p)$ be diagonalizable over a finite extension $K/\Q_p$ with pairwise distinct eigenvalues $\omega_1, \ldots, \omega_d$. Suppose that there exists an eigenbasis $\mathrm{x}_1, \ldots, \mathrm{x}_d$ with $\mathrm{M} \mathrm{x}_i=\omega_i \mathrm{x}_i$ for all $i$ such that $\mathrm{P}=[\mathrm{x}_1, \ldots, \mathrm{x}_d] \in \Gl_d(\mathcal{O}_K)$, where $\mathcal{O}_K$ is the ring of integers of $K$.
Then, for any $\mathrm{v} = \sum_{i=1}^d \alpha_i \mathrm{x}_i \in \Z_p^d$ with $\alpha_i \in \mathcal{O}_K^{\times}$, the $\Z_p$-submodule $L=\sum_{j=0}^{d-1} \Z_p\mathrm{M}^j\mathrm{v}$ of $\Z_p^d$ contains $p^N \Z_p^d$, where
\begin{equation*}
p^N=\max_{1 \leq i \leq d} \bigg \{ \prod_{j\neq i}|\omega_i-\omega_j|_p^{-1} \bigg \}.  \end{equation*}
Here, $| \cdot|_p$ denotes the $p$-adic absolute valuation normalized by $|p|_p=p^{-1}$.
\end{lemm}

\begin{proof}
Since $\mathrm{M}^j\mathrm{v}=\sum_{i=1}^d \alpha_i\omega_i^j \mathrm{x}_i$, we obtain $\mathrm{A}:= \bigl[\,\mathrm{v}, \mathrm{M}\mathrm{v}, \ldots, \mathrm{M}^{d-1}\mathrm{v}\,\bigr] = \mathrm{P}  \mathrm{D}  \mathrm{V}$, where $\mathrm{D}$ denotes the diagonal matrix $\mathrm{diag}(\alpha_1,\ldots,\alpha_d)$ and $\mathrm{V}=V(\omega_1, \ldots, \omega_d)$ is the Vandermonde matrix. In particular, $\mathrm{A}$ is invertible and $L$ has finite index in $\Z_p^d$. The inclusion $p^N \Z_p^d \subseteq L$ is equivalent to $p^{N} \mathrm{A}^{-1} \in M_d(\Z_p)$. If $d=2$, then $p^N=|\det(V)|^{-1}_p$, so that $p^N V^{-1} \in \Gl_2(\O_K)$. Hence, we have $p^N A^{-1} \in \Gl_2(\O_K) \cap M_2(\Z_p) = \Gl_2(\Z_p)$. Assume henceforth that $d\ge3$. Since the $(i,j)$th minor of the generic Vandermonde matrix $V(y_1, \ldots, y_d)$ vanishes after specializing to $y_a=y_b$ for any distinct $a, b \neq i$, each entry of $\mathrm{V}^{-1}$ has a denominator dividing $\prod_{j \neq i}(\omega_i - \omega_j)$ for some $1 \leq i \leq d$. Thus, we have $p^N \mathrm{V}^{-1} \in \Gl_d(\mathcal{O}_K)$, and the proof is completed exactly as in the case $d=2$.
\end{proof}

\begin{proof}[Proof of Proposition \ref{prop-Lie generator}]
For the case $\mathfrak{g}=\mathfrak{sl}_r$, let $k \in \NN$ satisfy $p^k > m(m+2)(m+3)$. Let $A$ and $C$ be the diagonal matrices with diagonal entries $\{\lambda_i\}_{1 \leq i \leq m}$ and $\{ \mu_j \}_{1 \leq j \leq n}$, respectively, as in Lemma \ref{lemm - eigenvalues} $(1)$. Set $\alpha =(A,C)^{\! \scriptscriptstyle +} \in \sl_r^+$. Then the set $$\mathcal{B}_{\mathfrak{sl}}:=\{(E_{i,j},0)^{\! \scriptscriptstyle -}, (0, E_{j,i})^{\! \scriptscriptstyle -} \mid 1 \leq i \leq m, \quad 1 \leq j \leq n \}$$ is an eigenbasis for $\mathrm{ad}_{\alpha}$ whose eigenvalues are $\{ \pm (\lambda_i - \mu_j)\}_{i,j}$. The eigenbasis $\mathcal{B}_{\mathfrak{sl}}$ also forms a natural $\Z_p$-basis of $\sl_r^{-} \cong \Z_p^{2mn}$.

Now define $\beta =  \sum_{1 \leq i \leq m, 1 \leq j \leq n} \big ( (E_{i,j},0)^{\! \scriptscriptstyle -} + (0, E_{j,i})^{\! \scriptscriptstyle -} \big ) \in \sl_r^-$. Then $\beta$ satisfies the hypothesis on $\mathrm{v}$ in Lemma \ref{lemm - Vandermonde} with respect to the eigenbasis $\mathcal{B}_{\mathfrak{sl}}$. Moreover, by Lemma \ref{lemm - eigenvalues} $(1)$, no difference of two distinct eigenvalues is divisible by $p^k$. Hence Lemma \ref{lemm - Vandermonde} yields
\begin{equation}\label{eq-inclusion-minus}
\tag{$\star$}
   L:= \sum_{j=0}^{\rk_{\Z_p}\mathfrak{sl}_r-1} \Z_p \mathrm{ad}_{\alpha}^j(\beta) \supseteq p^{(k-1)(\rk_{\Z_p} \mathfrak{sl}_r^- -1)} \mathfrak{sl}_r^-=p^{(k-1)(2mn-1)}\sl_r^-.
\end{equation}
Furthermore, Lemma \ref{lemm- conjugation} yields $[L \,\, L] \supseteq p^{2(k-1) (2mn -1)} \mathfrak{sl}_r^+$. Therefore, we have 
\begin{equation*}
    \langle \alpha, \beta \rangle \supseteq L + [L \,\, L] \supseteq p^{(k-1)(2mn -1)} \mathfrak{sl}_r^- + p^{2(k-1)(2mn - 1)} \sl_r^+ \supset p^{2(k-1)(2mn - 1)} \mathfrak{sl}_r.
\end{equation*}
\vskip 5pt
Let us now consider the case $\mathfrak{g}=\mathfrak{so}_r$. We first note that $\so_r^+=0$ if and only if $m=n=1$. In this case, we may take $\alpha=0$ and let $\beta$ be a $\Z_p$-generator of $\so_2^- \cong \Z_p$. Hence, we may assume that $m \geq 2$. 

Let $k \in \NN$ satisfy $p^k > 6(\lfloor m/2 \rfloor+1)\cdot \lfloor n/2 \rfloor+2 \lfloor m/2 \rfloor$, and let $\alpha=(A,C)^{\! \scriptscriptstyle +} \in \so_r^+$, where
\begin{equation*}
A= \sum_{i=1}^{\lfloor m/2 \rfloor} \lambda_i (E_{2i, 2i-1} - E_{2i-1, 2i}) \in M_m(\Z_p), \quad   C= \sum_{j=1}^{\lfloor n/2 \rfloor} \mu_j (E_{2j, 2j-1} - E_{2j-1, 2j}) \in M_n(\Z_p)  
\end{equation*}
with $\lambda_i$ and $\mu_j$ as in Lemma \ref{lemm - eigenvalues} $(2)$. Fix a square root $z$ of $-1$. Then, for each $1 \leq i \leq \lfloor m/2 \rfloor$, the vectors
\begin{equation*}
\hat{v}_i:= z e_{2i-1} + e_{2i} \in \Q_p(z)^m, \quad  \tilde{v}_i : = -z e_{2i-1} + e_{2i} \in \Q_p(z)^m  
\end{equation*}
are eigenvectors of $A$ with eigenvalues $\lambda_i z$ and $-\lambda_i z$, respectively. If $m$ is odd, then the standard basis vector $e_m \in \Q_p^m$ is an eigenvector of $A$ with eigenvalue $0$. Similarly, for each $1 \leq j \leq \lfloor n/2 \rfloor$, let $\hat{w}_j$ and $\tilde{w}_j$ denote the eigenvectors of $C$ with eigenvalues $\mu_j z$ and $-\mu_j z$, respectively. If $n$ is odd, then $e_n \in \Q_p^n$ is also an eigenvector of $C$. Since the eigenvalues of $A$ and $C$ are simple, the above eigenvectors form eigenbases $\mathcal{B}_A$ and $\mathcal{B}_C$ of $A$ and $C$, respectively.

By abuse of notation, we also write $(B)_{\mathfrak{so}}$ for the corresponding element of $\mathcal{O}_{\Q_p(z)} \otimes_{\Z_p} \mathfrak{so}_r^-$, whenever $B \in M_{m,n}(\mathcal{O}_{\Q_p(z)})$. Then $\mathcal{B}_{\mathfrak{so}}:=\{(x \cdot y^t)_{\mathfrak{so}} \mid x \in \mathcal{B}_A, y \in \mathcal{B}_C\} \subset \Q_p(z) \otimes_{\Z_p} \mathfrak{so}_r^-$ forms an eigenbasis of $\mathrm{ad}_\alpha$. Indeed, we have
\begin{equation*}
    \mathrm{ad}_{\alpha}((x \cdot y^t)_{\mathfrak{so}}) = (Ax \cdot y^t)_{\mathfrak{so}} - (x \cdot y^t C)_{\mathfrak{so}} = (Ax \cdot y^t)_{\mathfrak{so}} + (x \cdot (Cy)^t)_{\mathfrak{so}}.
\end{equation*} 
By Lemma~\ref{lemm - eigenvalues} $(2)$, the eigenvalues of $\mathrm{ad}_{\alpha}$ are pairwise distinct modulo $p^k$.

To apply Lemma \ref{lemm - Vandermonde}, we identify $\mathfrak{so}_r^-$ with $\Z_p^{mn}$ via the $\Z_p$-basis $\{(E_{ij})_{\mathfrak{so}}\}_{i,j}$.  Under this identification, the eigenbasis $\mathcal{B}_{\mathfrak{so}}$ satisfies the hypothesis on $\mathrm{P}$ in Lemma~\ref{lemm - Vandermonde}, since the change-of-basis matrices from the standard bases of $\Q_p^m$ (resp. $\Q_p^n$) to the eigenbasis $\mathcal{B}_A$ (resp. $\mathcal{B}_C$) belongs to $\Gl_m(\mathcal{O}_{\Q_p(z)})$ (resp. $\Gl_n(\mathcal{O}_{\Q_p(z)})$). We next observe that $\hat{v}_i + \tilde{v}_i = 2e_{2i} \in \Z_p^m$ and $\hat{w}_j+\tilde{w}_j = 2e_{2j} \in \Z_p^n$. Hence, the element $\beta=(S)_{\mathfrak{so}} \in \mathfrak{so}_r^-$, where
\begin{equation*}
    S= \bigg ( \sum_{i=1}^{\lfloor m/2 \rfloor} (\hat{v}_i + \tilde{v}_i ) + (m-2\lfloor m/2 \rfloor ) e_m \bigg) \cdot \bigg ( \sum_{j=1}^{\lfloor n/2 \rfloor} (\hat{w}_j + \tilde{w}_j ) + (n-2\lfloor n/2 \rfloor ) e_n \bigg)^t \in M_{m,n}(\Z_p)
\end{equation*}
satisfies the hypothesis on $\mathrm{v}$ in Lemma \ref{lemm - Vandermonde}. The proof now proceeds exactly as in the $\mathfrak{sl}_r$ case.
\end{proof}

\begin{proof}[Proof of Proposition \ref{prop - Kuranashi - p}]
Let $\mathfrak{g}$ be either $\mathfrak{sl}_r$ or $\mathfrak{so}_r$. Let $\alpha$ and $\beta$ be as in the proof of Proposition \ref{prop-Lie generator} and set $\alpha'=p\alpha$ and $\beta'=p\beta$. Since $\mathrm{ad}^j_{\alpha'}(\beta')=p^{j+1} \mathrm{ad}_{\alpha}^j(\beta)$, by  \eqref{eq-inclusion-minus}, we obtain $$L':=\sum^{\rk_{\Z_p} \mathfrak{g}^--1}_{j=0} \Z_p \mathrm{ad}_{\alpha'}^j(\beta') \supseteq p^{\rk_{\Z_p} \mathfrak{g}^-} \cdot p^{(k-1) \cdot (\rk_{\Z_p} \mathfrak{g}^--1)} \mathfrak{g}^- = p^{k \rk_{\Z_p}\mathfrak{g}^- - k +1} \mathfrak{g}^-.$$ The claim now follows from $\langle \alpha', \beta' \rangle \supset [L' \,L'] + L'$.    
\end{proof}


\section{Proof of the main theorems}

With the Lie algebra generators established in the previous section, we are now ready to prove the main theorems.

\begin{proof}[Proof of Theorem \ref{theo-sl}]
Let $p \geq 3$ be a prime and $r \geq 2$ an integer. For $a \in \Z$ satisfying  $\epsilon(r) \leq a \leq r-2$ and $a \equiv r \pmod{2}$, there exists a partition $r=m+n$ such that $m-n=a$. Let $k$ be an integer satisfying $p^k > m(m+2)(m+3)$ and set $\gamma=4kmn-2k+2$.

Let $\sl_r^{\pm}$ be the $\pm1$-eigenspaces of $\sl_r$ with respect to the conjugation action by $d_{m,n}$. By Proposition \ref{prop - Kuranashi - p} $(1)$, there exist $\alpha \in p \sl_r^+$ and $\beta \in p \sl_r^-$ such that the $\Z_p$-Lie subalgebra $\langle \alpha, \beta \rangle$ of $\sl_r^1$ contains $\sl^{\gamma}_r$. Set $g_1=\exp(\alpha), g_2=\exp(\beta)$, and let $H$ be the closed subgroup of $\Sl^1_r(\Z_p)$ generated by $g_1$ and $g_2$. Then $(g_1, g_2)$ is a Kuranishi pair of $\Sl_r(\Z_p)$. Moreover, we have $$H \supseteq\exp(\sl^{\gamma}_r)=\Sl_r^{\gamma}(\Z_p).$$

Now take $K=\Q(\sqrt{-p})$. Since $p\nmid h_K$ by Example~\ref{example}$,$ and $2$ is ramified in $K/\Q$ if and only if $p\equiv3\pmod4$, Theorem~\ref{theorem_4.3_bis} yields statements $(1)$ and $(2)$.

The first assertion of $(3)$ follows from the fact that, for any complex conjugation $c \in G_{\Q}$, we have $\rho \otimes \omega(c) = -\rho(c)$, and hence $${\rm tr}(\rho \otimes \omega(c))=-{\rm tr}(\rho(c))=-{\rm tr}(A)=-a.$$ For the second assertion, note that every open subgroup of $\Sl_r(\Z_p)$ has finite abelianization. Hence for some open subgroup $U$ of $\Gal(K_p/\Q)$, we have $(\rho \otimes \omega)(U) \cong \rho(U) \times \Z_p$. It follows that $\mathrm{im}(\rho \otimes \omega) \subset \Gl_r(\Z_p)$ has dimension $r^2$ as a $p$-adic analytic group. Since this coincides with the dimension of $\Gl_r(\Z_p)$, the image $\mathrm{im}(\rho \otimes \omega)$ is open in $\Gl_r(\Z_p)$.
 \end{proof}

\begin{proof}[Proof of Theorem \ref{theo-so}]
The proof is obtained by adapting the argument of Theorem \ref{theo-sl} to $\so_r$. Let $m+n=r$ with $m-n=a$, and let $k \in \NN$ satisfy
\begin{equation*}
p^k> \frac{3n(m+2)}{2} + m.   
\end{equation*}
Set $\gamma:=2kmn-2k+2$. By Proposition~\ref{prop - Kuranashi - p} $(2)$ applied to $d_{m,n}$, there exist $\alpha \in p \so_r^+$ and $\beta \in p \so_r^-$ such that the $\Z_p$-Lie subalgebra $\langle \alpha, \beta \rangle$ contains $\so^{\gamma}_r$. Passing to the exponential map, the closed subgroup of $\mathrm{SO}^1_r(\Z_p)$ generated by $\exp(\alpha)$ and $\exp(\beta)$ contains the open subgroup $\mathrm{SO}_r^\gamma(\Z_p)$ of $\SO_r(\Z_p)$. The remainder of the proof is identical to that of Theorem \ref{theo-sl}. Note that $A=d_{m,n}$ has order $2$ whereas $\SO_r^1(\Z_p)$ is torsion-free, being uniform. Hence, $\SO_r^1(\Z_p) \rtimes A$ embeds naturally into $\OO_r(\Z_p)$, and the image of Galois representation $\rho : G_\Q \to \SO_r^1(\Z_p) \rtimes \langle A \rangle$ constructed above is contained in $\OO_r(\Z_p)$.
\end{proof}

\begin{rema}
In all previous works on non-geometric Galois representations, the resulting Galois representations are regular at infinity. For the constructions in \cite{Ray, Maletto, Ray_2, Tang}, this can be checked from the description of the residual mod $p$ representation: see \S2.1 of \cite{Ray}, Theorem~4.3(1) of \cite{Maletto}, Theorem~3.3(2) of \cite{Ray_2}, and Condition~(3) of Theorem~3.10 in \cite{Tang}. In \cite{Greenberg}, the action of $\Delta'$ on $H$ is given, in our notation, by conjugation by a diagonal matrix whose diagonal entries alternate in parity (cf.~Proposition~6.1). A similar argument appears in \S3.4 of \cite{Cornut-Ray}, while the corresponding argument in \cite{Maire} can be found in \S4.2.
\end{rema}


\bibliography{ref}
\bibliographystyle{amsplain}

\end{document}